\documentclass[a4paper,12pt,reqno]{amsart}
\usepackage[utf8]{inputenc}
\usepackage{amssymb}
\usepackage{amsthm}

\usepackage{xcolor}

\numberwithin{equation}{section}

\newtheorem{thm}{Theorem}[section]
\newtheorem{lemma}[thm]{Lemma}
\newtheorem{cor}[thm]{Corollary}
\newtheorem{prop}[thm]{Proposition}
\theoremstyle{definition}

\newtheoremstyle{miremark}{}{}{}{}{\bfseries}{.}{ }{}%
\theoremstyle{miremark}
\newtheorem{rmk}[thm]{Remark}

\begin{document}
\title{
  Superdifferential of the Takagi function   }

\author{Juan Ferrera}
\address{IMI, Departamento de An{\'a}lisis Matem{\'a}tico y
  Matem\'atica Aplicada,
Facultad Ciencias Ma\-te\-m{\'a}ticas, Universidad Complutense, 28040, Madrid, Spain}
\email{ferrera@mat.ucm.es}

\author{Javier G\'omez Gil}
\address{Departamento de An{\'a}lisis Matem{\'a}tico y
  Matem\'atica Aplicada,
Fa\-cul\-tad Ciencias Matem{\'a}ticas, Universidad Complutense, 28040, Madrid, Spain}
\email{gomezgil@mat.ucm.es}

\subjclass[2010]{Primary 26A27; Secondary 26A24, 49J52}

\keywords{Takagi function, Superdifferential }

\begin{abstract}
The Takagi function is a classical example of a continuous nowhere differentiable
function.  It has empty subdifferential except in a countable
set where its subdifferential is $\mathbb{R}$. 
In this paper we characterize its superdifferential.
\end{abstract}

\maketitle
\section{Introduction and preliminary}

The Takagi function is probably the easiest example of
a continuous nowhere derivable function.
Less known than Weierstrass function,
it was introduced in 1903 by T. Takagi  (see \cite{Takagi})  and ever since it has caught the interest
of mathematicians. 
For an extensive information about this function, see the surveys 
\cite{AK} and \cite{Lagarias}.

The Takagi function, $T:\mathbb{R}\to \mathbb{R}$, is usually defined as 
\begin{equation*}
T(x)=\sum_{n=0}^{\infty}\frac{1}{2^n}\phi (2^nx)
\end{equation*}
where $\phi (x)=\text{dist}(x,\mathbb{Z})$, the distance from $x$ to the nearest integer.

If $D$ denote the set of all dyadic
real numbers and  we decompose it as an increasing union of sets
\begin{equation*}
 D_n=\bigg\{ \frac{k}{2^{n-1}}:k\in \mathbf{Z}\bigg\}, \quad n=1,2,\dotsc
\end{equation*}
then we may also define the Takagi function as:
\begin{equation}
T(x)=\sum_{k=1}^{\infty} g_k(x)=\lim_n G_n(x), \quad(x\in\mathbb{R}),\label{eq:1}
\end{equation}
where $g_k(x)=\min (|x-y|:\, y\in D_k)$ is the distance of the point
$x$ to the closed set $D_k$,  and $G_n=g_1+\dots +g_n$.

Now we will introduce some non smooth definitions and results.
As a general reference for these concepts you may use  
\cite{F} or \cite{RW}.

For an upper semicontinuous function $f:\mathbb{R}\longrightarrow
\mathbb{R}$ and a point $x\in \mathbb{R}$, we define the
 Fréchet superdifferential of $f$ at $x$, $\partial^+ f(x)$, as the
 set of $\xi\in \mathbb{R}$ that satisfy
\begin{equation*}
\limsup_{h\to 0}\frac{f(x+h)-f(x)- \xi h}{|h|}\leq 0.
\end{equation*}
The function $f$ is said superdifferentible at $x$ if
$\partial^{+}f(x)\neq \varnothing$.

In a similar way, for a lower semicontinuous function $f:\mathbb{R}\longrightarrow
\mathbb{R}$ and a point $x\in \mathbb{R}$, the
 Fréchet subdifferential of $f$ at $x$, $\partial f(x)$, is defined as the set of $\xi\in\mathbb{R} $ that satisfy
\begin{equation*}
\liminf_{h\to 0}\frac{f(x+h)-f(x)- \xi h}{|h|}\geq 0.
\end{equation*}
If $\partial f(x)\neq \varnothing$, then $f$ is said subdifferentiable at $x$.
It is immediate that  $\partial^{+} f (x) = -\partial(-f )(x)$.

A continuous function $f$ is derivable at $x$ if and only if both $\partial f (x)$ and
$\partial^{+} f (x)$ are nonempty. In this case, 
$\partial f (x)=\partial^{+} f(x) = \{f'(x)\}$. 

 Next proposition characterizes the superdifferential of an upper semicontinuous function $f:\mathbb{R}\longrightarrow
 \mathbb{R}$ at a point $x$  in terms of the Dini derivatives:
\begin{equation*}
  \begin{gathered}[t]
    d_{-}f(x)=\liminf_{h\uparrow 0}\frac{f(x+h)-f(x)}{h}\\
D^{+}{f}(x)=\limsup_{h\downarrow 0}\frac{f(x+h)-f(x)}{h}.
  \end{gathered}
\end{equation*}

\begin{prop}
  \label{sec:introduction-3}
  An upper semicontinuous function  $f:\mathbb{R}\longrightarrow \mathbb{R}$ is
  superdifferentiable at a point $x$ if, and only if, 
  \begin{equation*}
D^+f(x)\leq
  d_-f(x)\quad\text{ and}\quad  \left[D^+f(x), \,d_-f(x)\right]\cap
  \mathbb{R}\neq\varnothing.
\end{equation*}
In this case, $\partial^+ f(x)=[D^+f(x),
  \,d_-f(x)]\cap \mathbb{R}$.
\end{prop}

\section{Main Theorem}
\label{sec:main-theorem}

Gora and Stern, see \cite{GS},  showed that $T$ has empty Fréchet subdifferential at every $x\not\in D$,
meanwhile its subdifferential is $\mathbb{R}$ at every $x\in D$. In
this section we
are going to characterize the superdifferential of $T$.

If $x\not\in D$,
for all $n$ there exist $x_n,y_n\in D_n$ such that
$x\in (x_n,y_n)$ and $(x_n,y_n)\cap D_n=\varnothing$. We denote
$c_n=\frac12\left(x_n+y_n\right)$. 
It is obvious that $c_n\in D_{n+1}$. Moreover, if  $x>c_n$ then $c_n=x_{n+1}$,
meanwhile  $c_n=y_{n+1}$ otherwise.

On the other hand,  for $k< n$ we have that either
$(x_n,y_n)\subset (x_k,c_k)\subset(x_k,y_k)$ or $(x_n,y_n)\subset (c_k,y_k)\subset(x_k,y_k)$.

\begin{prop}\label{sec:introduction-1}
  If $x\not\in D$ then
  \begin{equation*}
    \begin{aligned}
      d_-T(x)& \leq \liminf_n G'_{n}(x)+1, & \\ \ D^+T(x) & \geq
      \limsup_n G'_{n}(x)-1.&
    \end{aligned}
  \end{equation*}
\end{prop}

\begin{proof}  
   Let $x_n, y_n, c_n$ be as before.  
  The sequences  $(x_n)_n$ and $(y_n)_n$ converge to $x$. For every
  $n$ let $x'_n=2x_n-x$. As $g_k$ is an even
  periodic function of period $2^{-k+1}$, we have that
  $g_k(x'_k)=g_k(x)$ 
  for all $k\geq n-1$ and therefore
  \begin{equation}\label{eq:3}
    T(x'_n)-T(x)=\sum_{k<n-1} \left(g_k(x'_n)-g(x)\right)
  \end{equation}
for all $n$.

If $x_{n-1}<x_n$ then $c_{n-1}=x_n$ and $y_{n-1}=y_n$ and therefore
$g'_{n-1}(x)=-1$. Moreover,  as
$x'_n\in \left(x_{n-1},x_n\right)\subset
\left(x_{n-1},y_{n-1}\right)$, then  either
 $ x'_n,x\in (x_k,c_k) $ or $x'_n,x\in
(c_k,y_k)$ for every $k<n-1$,  therefore
\begin{equation*}
  g_k(x'_n)-g_k(x)=g'_k(x)(x'_n-x)
\end{equation*}
for all $k<n-1$ 
and thus
\begin{equation}
  \label{eq:13}
  \frac{T(x'_n)-T(x)}{x'_{n}-x}
  =G'_{n-2}(x)
  =G'_{n-1}(x)+1.
\end{equation}

If $x_{m}<x_{m+1}=\dotsb=x_{n-1}=x_{n}$ for some $m<n-1$, then
$x'_{m+1}=\dotsb=x'_{n-1}=x'_{n}$.  Hence, by \eqref{eq:13},
\begin{equation}
  \label{eq:5}
   \frac{T(x'_n)-T(x)}{x'_{n}-x}= \frac{T(x'_{m+1})-T(x)}{x'_{m+1}-x}
  =G'_{m}(x)+1 <G'_{n-1}(x)+1
\end{equation}
since $g'_k(x)=1$ for $k=m+1,\dots, n-1$.

It follows from
\eqref{eq:13} and \eqref{eq:5}
that
\begin{equation}\label{eq:8}
  d_-T(x)
  \leq
  \liminf_{n} \frac{T(x'_{n})-T(x)}{x'_{n}-x}  \leq 
 \liminf_n G'_{n}(x)+1.
\end{equation}

As $T$ is an even function
\begin{equation}
\begin{aligned}
  D^+T(x)=&  \limsup_{h\downarrow 0} \frac{T(x+h)-T(x)}{h} \\
  = &\limsup_{h\uparrow 0} \frac{T(-x+h)-T(-x)}{-h} \\
  =&-\liminf_{h\uparrow 0} \frac{T(-x+h)-T(-x)}{h}=-d_-T(-x)
\end{aligned}\label{eq:4}
\end{equation}
and therefore
\begin{align*}
  D^+T(x)= &-d_-T(-x)\geq - \liminf_n G'_{n}(-x)-1  \\ = &\limsup_n G'_{n}(x)-1
\end{align*}
because the functions $G_n$ are even also. 
\end{proof}

\begin{cor}
  \label{sec:introduction-2}
  Let $x\not\in D$ be.
  If  one of two limits, $\limsup_n G'_{n}(x)$ or
$\liminf_n G'_{n}(x)$, is infinite or  both limits are finite and
\begin{equation*}
\limsup_n G'_n(x)- \liminf_n G'_n(x)> 2
\end{equation*}
 then $\partial^+T(x)=\varnothing$.
\end{cor}

\begin{proof}
It is immediate from Proposition \ref{sec:introduction-3} that
$\partial^+T(x)=\varnothing$   provided that either  $\limsup_n G'_{n}(x)=+\infty$ or
$\liminf_n G'_{n}(x)=-\infty$.

If both limits are finite and
$\limsup_n G'_n(x)- \liminf_n G'_n(x)> 2$ then
\begin{equation*}
  D^+T(x)-d_-T(x)\geq \limsup_n G'_n(x)- \liminf_n G'_n(x)-2>0
\end{equation*}
and  therefore, by Proposition \ref{sec:introduction-3} again,  $\partial^+T(x)=\varnothing$. 
\end{proof}

Now  we are going to study the situation when  both  $\limsup_n G'_{n}(x)$ and
$\liminf_n G'_{n}(x)$ are finite  and
\begin{equation}\label{eq:7}
  \limsup_n G'_n(x)- \liminf_n G'_n(x)\leq 2.
\end{equation}

In this case, as $|g'_k(x)|=1$ for all $k$, the sequence $(G'_n(x))_n$ does not
converge in $\mathbb{R}$ and  $\limsup_n G'_n(x)- \liminf_n G'_n(x)\geq
1$ and,  as both limits are integer numbers, this difference can only take the values 1 and 2.

\begin{lemma}
  \label{sec:main-theorem-2}
  Let $x\in\mathbb{R}$ be. If $x=k+\sum_{n=1}^\infty a_n2^{-n}$ for
  some $k\in \mathbb{Z}$ and $a_n\in\{0,1\}$, $n=1,2,\dotsc$ then
  \begin{equation*}
      g_n(x)= \frac{a_{n}}{2^{n}}+\left(1-2a_{n}\right)\sum_{k\geq n+1}\frac{a_k}{2^k}
\end{equation*}
for all $n=1,2,\dotsc$.
\end{lemma}
\begin{proof}
Observe that
\begin{equation*}
 \begin{split}
      g_n(x)= &\left\{
        \begin{aligned}
          & \sum_{k=n+1}^\infty \frac{a_k}{2^k},\qquad & \text{ if
            $a_{n}=0$,} \\
          &\frac1{2^{n}}- \sum_{k=n+1}^\infty \frac{a_k}{2^k}, & \text{
            if $a_{n}=1$};
        \end{aligned}
      \right. \\[1em]
    \end{split}
\end{equation*}
\end{proof}

Note that if $x\not \in D$ then $g'_n(x)=1$ if and only if
$a_{n}=0$. In consequence, if $x\not\in D$,  $g'_{n}(x)=1-2a_{n}$.

\begin{lemma}
  \label{sec:main-theorem-3}
  If $x\in\mathbb{R}$ then
  \begin{equation}
    \label{eq:17}
    g_n(x)+g_{n+1}(x)\leq \frac1{2^{n}}
  \end{equation}
  for all $n$.

  Moreover, for  $x\not\in D$, $g_n(x)+g_{n+1}(x)=2^{-n}$ if and
  only if $g'_{n}(x)\neq g'_{n+1}(x)$.
\end{lemma}

\begin{proof}
  With the notations of  Lemma \ref{sec:main-theorem-2},
  \begin{multline}\label{eq:15}
    g_n(x)+g_{n+1}(x)= \\ =\frac{a_{n}}{2^{n}}+\left(1-2a_{n}\right)\sum_{k\geq
          n+1}\frac{a_k}{2^k}+
        \frac{a_{n+1}}{2^{n+1}}+\left(1-2a_{n+1}\right)\sum_{k\geq
          n+2}\frac{a_k}{2^k} \\ =
        \frac1{2^{n}}\left(a_{n}+a_{n+1}- a_{n}a_{n+1}\right) -
        2\left(a_{n}+a_{n+1}-1\right)\sum_{k\geq
          n+2}\frac{a_k}{2^k}.
      \end{multline}
      If $a_{n}+a_{n+1}=0$ then $a_n=a_{n+1}=0$ and therefore
    \begin{equation*}
      g_n(x)+g_{n+1}(x)=2\sum_{k\geq
          n+2}\frac{a_k}{2^k} \leq 2 \sum_{k\geq
          n+2}\frac{1}{2^k}= \frac1{2^{n}}
    \end{equation*}
  with strict inequality  if $x\not\in D$.
If $a_{n}+a_{n+1}\geq 1$
 then $a_{n}+a_{n+1}- a_{n}a_{n+1}=1$ and
\eqref{eq:17} is an immediate consequence of \eqref{eq:15}. Moreover, if
$x\not\in D$,  the equality
holds if and only if $a_{n}+a_{n+1}-1=0$. But $a_{n}+a_{n+1}-1=0$
if and only if
$a_{n}\neq a_{n+1}$.
\end{proof}

\begin{prop}
  \label{sec:main-theorem-1}
   Let $x\not\in D$ be. 
   If $\liminf_n G'_n(x)$ and $\limsup_n G'_n(x)$ are finite and
   \begin{equation*}
     \limsup_n G'_n(x)-\liminf_n G'_n(x)\in \{1,2\}
   \end{equation*}
   then there exists $m\geq 1$ such
    that $a_{m+2i}+a_{m+2i+1}=1$ for all $i\geq 0$ and 
  \begin{equation}
    \label{eq:10}
     d_-T(x) =
 \liminf_n G'_{n}(x)+1     
  \end{equation}
and
  \begin{equation}
    \label{eq:11}
    D^+T(x) = \limsup_n G'_{n}(x)-1.
  \end{equation}
\end{prop}

\begin{proof}
  Let  $S=\limsup_n G'_n(x)$ and $I= \liminf_n G'_n(x)$ be. There
  exists an integer $m> 0$ such that 
   $I\leq G'_n(x)\leq
  S$ for all $n\geq m-1$.  We can assume, without loss of
  generality, that $G'_{m}(x)=I$. This implies in particular that
  $G'_{m-1}(x)=I+1$ and therefore that
  $|G'_{n}(x)-G'_{m-1}(x)|\leq 1$ for all $n\geq m$ .  As $g'_{j}(x)+g'_{j+1}(x)$
  is an even integer for al $j\geq 0$, then
  \begin{equation}\label{eq:12}
    G'_{m+2i+1}(x)-G'_{m-1}(x)=
    \sum_{j=1}^{i} \left( g'_{m+2j}(x)+g'_{m+2j+1}(x)\right)=0
  \end{equation}
for all $i\geq 0$. In particular
\begin{multline}
\label{eq:9}
  g'_{m+2i+1}(x)+g'_{m+2i}(x) \\
  =
  G'_{m+2i+1}(x)-G'_{m-1}(x)-\left(G'_{m+2i-1}(x)-G'_{m-1}(x)
  \right)=0
\end{multline}
for all $i\geq 0$.

By  Lemma \ref{sec:main-theorem-3} 
\begin{equation}\label{eq:18}
  g_{m+2i+1}(x)+g_{m+2i}(x)=\frac1{2^{m+2i}},
\end{equation}
for all $i\geq 0$,
and thus, again by  Lemma \ref{sec:main-theorem-3}, 
\begin{equation*}
  g_{m+2i+1}(x)+g_{m+2i}(x)-
  \left(g_{m+2i+1}(x')+g_{m+2i}(x')\right)\geq   0.
\end{equation*}
for all $x'\in \mathbb{R}$.
This implies that
\begin{multline}\label{eq:20}
  \sum_{j\geq m+2p}\left(g_j(x)-g_j(x')\right) \\ =
  \sum_{i\geq p}
  \left[g_{m+2i+1}(x)+g_{m+2i}(x)-
    \left(g_{m+2i+1}(x')+g_{m+2i}(x')\right) \right] \geq 0
\end{multline}
for all $p\geq 0$.

  Let 
  $x_n, y_n, c_n$ be as before.
 If $x'\in (x_{m},x)$ then there  exists
 $n>m$ such that $x_{n-1}< x'\leq x_n$.  
This implies that $x_n=c_{n-1}$ and thus $g'_{n-1}(x)=-1$ and 
$G'_{n-2}(x)=G'_{n-1}(x)+1\geq I+1$. Therefore
 \begin{align}\label{eq:19}
   \frac{T(x')-T(x)}{x'-x}
   = &G'_{n-2}(x) +\frac1{x'-x}\,\sum_{j\geq
       n-1}\left(g_j(x')-g_j(x)\right)  \notag  \\
   \geq & I+1
   +\frac{1}{x-x'}\,\sum_{j\geq n-1}\left(g_j(x)-g_j(x')\right). 
 \end{align}

  If $n=m+2p+1$ for some $p\geq 0$ 
  then, 
  by \eqref{eq:20},
  \begin{equation*}
   \frac{T(x')-T(x)}{x'-x}
   \geq I+1.
 \end{equation*}
 
  If $n=m+2p$ for some $p> 0$ then, by \eqref{eq:12},
  $G'_{n-2}(x)=I+2$, and  as
  \begin{equation*}
    \left|g_{n-1}(x')-g_{n-1}(x)\right|\leq |x'-x|
  \end{equation*}
 then
\begin{multline*}  
   \frac{T(x')-T(x)}{x'-x}
   =G'_{n-2}(x) +\frac1{x'-x}\,\sum_{j\geq n-1}\left(g_j(x')-g_j(x)\right)  \\
   =I+2+\frac{g_{n-1}(x')-g_{n-1}(x)}{x'-x}
   +\frac1{x'-x}\,\sum_{j\geq n}\left(g_j(x')-g_j(x)\right) \\ 
   \geq I+1    
 \end{multline*}
by \eqref{eq:20}.
In both cases, by applying Proposition \ref{sec:introduction-1}, we obtain
\eqref{eq:10}. Relation \eqref{eq:11} follows from the equality $D^+T(x)=-d_-T(-x)$.
\end{proof}

\begin{thm}
  \label{sec:main-theorem-4}
  Let $x \in \mathbb{R}$ be, $x=k+\sum_{n=1}^\infty a_n2^{-n}$ for
  some $k\in \mathbb{Z}$, with $a_n\in\{0,1\}$, $n=1,2,\dotsc$.
  \begin{enumerate}
  \item \label{item:1} If $x\in D$ then $\partial^+T(x)=\varnothing$.
  \item \label{item:2} If $x\not\in D$ and there exists $m\in \mathbb{Z}$, $m\geq 1$, such
    that $a_{n}+a_{n+1}=1$ for all $n\geq m$, then
    \begin{equation*}
      \partial^+T(x)=\left\{
          \begin{aligned}
           &  m-1- 2\sum_{j=1}^{m-1}  a_j+\,[0,\,1\,]\quad &
             \text{ if $a_{m}=0$,} 
             \\[.5em]
           & m-1- 2\sum_{j=1}^{m-1}  a_j+\,[\,-1,\,0\,] &\text{ if $a_{m}=1$.} 
          \end{aligned}
\right.
\end{equation*}
\item \label{item:3} If $x\not\in D$ does not satisfy (\ref{item:2})
  but there  exists $m\geq 1$, such
    that $a_{m+2i}+a_{m+2i+1}=1$ for all $i\geq 0$, then
    \begin{equation*}
\partial^+T(x)=\left\{m-1- 2\sum_{j=1}^{m-1} a_j, \right\}
\end{equation*}
\item If $x\not\in D$ and does not verify  (\ref{item:2}) and        
  (\ref{item:3}) then  $\partial^+T(x)=\varnothing$.
  \end{enumerate}
\end{thm}

\begin{proof}

If $x\in D_{n}$ then for all integer  $p\geq n+1$
  \begin{equation*}
    T\left(x+\frac1{2^{p}}\right)-T(x)= \sum_{j=1}^{p} g'_{j}(x)\frac1{2^{p}}  = 
    \left(\sum_{j=1}^{n} g'_{j}(x)+ (p-n)\right)\,\frac1{2^{p}}
  \end{equation*}
and therefore
  \begin{equation*}
    D^{+}T(x)\geq \limsup_{p}\frac{T(x+2^{-p})-T(x)}{2^{-p}}=+\infty.
  \end{equation*}
It follows from \ref{sec:introduction-3} that
$\partial^{+}T(x)=\varnothing$.

If $x\not\in D$ and there exists $m\in \mathbb{Z}$, $m\geq 1$, such
    that $a_{n}+a_{n+1}=1$ for all $n\geq m$, then $g'_{n}(x)+g'_{n+1}(x)=0$ 
    for all $n\geq m$ and therefore
    \begin{equation*}
        G'_{m+2i+1}(x)=G'_{m-1}(x)+\sum_{j=0}^{i}
        \left(g'_{m+2j}(x)+g'_{m+2j+1}(x)\right) = G'_{m-1}(x)
    \end{equation*}
and
\begin{equation*}
  \begin{split}
    G'_{m+2i}(x)= &G'_{m}(x)+\sum_{j=1}^{i}
    \left(g'_{m+2j-1}(x)+g'_{m+2j}(x)\right) \\= &G'_{m}(x)=G'_{m-1}(x)+g'_{m}(x)
  \end{split}
\end{equation*}
for all $i\geq 1$.

If $a_{m}=0$ then $g'_{m}(x)=1$ and
\begin{equation*}
\liminf_{n} G'_{n}(x)= G'_{m-1}(x), \quad \limsup_{n} G'_{n}(x)=G'_{m-1}(x)+1.
\end{equation*}
Hence, by Proposition \ref{sec:main-theorem-1} and Proposition \ref{sec:introduction-3}, 
\begin{equation*}
  \partial^{+}T(x)=G'_{m-1}(x)+[0,1]=\sum_{j=1}^{m-1}(1-2a_{j})+[0,1].
\end{equation*}

Alternatively, if $a_{m}=1$ then  $g'_{m}(x)=-1$ and
\begin{equation*}
\liminf_{n} G'_{n}(x)= G'_{m-1}(x)-1, \quad \limsup_{n} G'_{n}(x)=G'_{m-1}(x)
\end{equation*}
and, by Propositions \ref{sec:main-theorem-1} and \ref{sec:introduction-3}, 
\begin{equation*}
  \partial^{+}T(x)=G'_{m-1}(x)+[-1,0]=\sum_{j=1}^{m-1}(1-2a_{j})+[-1,0].
\end{equation*}

In order to see (\ref{item:3}):
if $x$ satisfies $g'_{m+2i}(x)+g'_{m+2i+1}(x)=0$ for all
$i\geq 0$, then 
\begin{equation*}
  G'_{m+2n-1}(x)=G'_{m-1}(x)+\sum_{i=0}^{n}
        \left(g'_{m+2i}(x)+g'_{m+2i+1}(x)\right)=G'_{m-1}(x)
\end{equation*}
and
\begin{equation*}
  G'_{m+2n}(x)=G'_{m+2n-1}(x)+g'_{m+2n}(x)=G'_{m-1}(x)+g'_{m+2n}(x).
\end{equation*}
for all $n\geq 1$. 

On the other hand 
$\liminf_{n}g'_{m+2n}(x)=-1$ and $\limsup_{n}g'_{m+2n}(x)=1$ 
since $x$ does not satisfy (\ref{item:2}). Therefore
\begin{equation*}  
  \begin{gathered}
    \liminf_{n} G'_{n}(x)=G'_{m-1}(x)-1\\
    \limsup_{n} G'_{n}(x)=G'_{m-1}(x)+1
  \end{gathered}
\end{equation*} 
It follows from Propositions \ref{sec:introduction-3} and
\ref{sec:introduction-1} that $\partial^{+}T(x)=\{G'_{m-1}(x)\}$.

Finally, if $x\not\in D$  does not verify  (\ref{item:2}) and        
  (\ref{item:3}) the result follows from Proposition \ref{sec:main-theorem-1} and
Corollary  \ref{sec:introduction-2}.
\end{proof}

\begin{rmk}
The set $A$ of points that satisfy condition $(2)$ is a countable dense subset of $\mathbb{R}$
\end{rmk}
\begin{proof}
$A$ is clearly countable, since it is the set of points where a function on $\mathbb{R}$ is superdifferentiable,
with superdifferential not a singleton. However, directly, it is not difficult to see that
\begin{equation*}
  A=\bigcup_{m=1}^{\infty}\left(\frac43\,\frac1{2^{m}}+D_{m}\right)
\end{equation*}
which give us the result.
\end{proof}

In order to study the set of points where the Takagi function is superdifferentiable,
we will take advantage of the following result, due to Kahane, (see \cite{K} and  \cite{AK}),
which describes the set $\mathcal{M}$ where the Takagi function attains its maximum value, which is 
$\frac{2}{3}$.

\begin{thm}\label{Kahane}
$\mathcal{M}$ is a set of Hausdorff dimension $\frac{1}{2}$, and consists of all the points $x$ with 
binary expansion satisfying $a_{2n-1}+a_{2n}=1$ for every $n\geq 1$.
\end{thm}

From Theorem \ref{sec:main-theorem-4} we deduce that  
$\partial ^+T(x)\neq \varnothing$ if and only if $x$ satisfies either condition $(2)$ or $(3)$; in other words,
if and only if there exists $m\geq 1$, such
that $a_{m+2i}+a_{m+2i+1}=1$ for every $i\geq 0$. With the aim of studying these points we introduce the following 
notations: 
\begin{equation*}
\hat{x}_m=\sum_{j=1}^{m-1}\frac{a_j}{2^j}, \quad
\tilde{x}_m=\sum_{j=m}^{\infty}\frac{a_j}{2^j},\quad
\widehat{T}=\sum_{j=1}^{m-1}g_j, \quad \widetilde{T}=\sum_{j=m}^{\infty}g_j
\end{equation*}
and
\begin{equation*}
c_x=m-1- 2\sum_{j=1}^{m-1} a_j.
\end{equation*}

We have the following result.
\begin{cor}
The set $\mathcal{A}$ of points $x$ satisfying  $\partial ^+T(x)\neq \varnothing$
 is an uncountable Lebesgue null set with Hausdorff dimension $\frac{1}{2}$.
Moreover, for every $x\in \mathcal{A}$, the function $T(y)-c_x(y-x)$ attains a 
local maximum at $x$.
\end{cor}
\begin{proof}
For the first part it is enough to observe that we can write
\begin{align*}
\mathcal{A}=& \bigcup_{m=1}^{\infty}\Bigl\{ x: a_{m+2i}+a_{m+2i+1}=1 \
              \text{for all} \ i\geq 0 \Bigr\} \\= &
\bigcup_{m=1}^{\infty}\Bigl( D_{m}+\frac{1}{2^{m-1}} \mathcal{M} \Bigr) 
\end{align*}
and then apply Theorem \ref{Kahane}.

For the second part, if $x\in A$ then $x=\hat{x}_m+\frac{1}{2^{m-1}}z$  
with $\hat{x}_m\in D_{m}$ and $z\in \mathcal{M}$. Observe that, as 
$\widetilde{T}$ is $2^{-(m-1)}$-periodic,
\begin{align*}  \widetilde{T}(x)=&\widetilde{T}\left(\frac{1}{2^{m-1}}z\right)=\sum_{j=m-1}^{\infty}\frac{1}{2^j}\phi
  (2^{j-m+1}z) \\= & \sum_{k=0}^{\infty}\frac{1}{2^{k+m-1}}\phi (2^kz)=
  \frac{1}{2^{m-1}}T(z).
\end{align*}
Similarly, for every $y$,
\begin{equation*}
\widetilde{T}(y)=\frac{1}{2^{m-1}}T(z_y)
\end{equation*}
where $z_y=2^{m-1}\tilde{y}_m$. Hence $\widetilde{T}(x)\geq \widetilde{T}(y)$ for every $y$, which 
implies
\begin{equation*}
T(x)=\widehat{T}(x)+\widetilde{T}(x)\geq \widehat{T}(x)+\widetilde{T}(y)=\widehat{T}(y)-c_x(y-x)+\widetilde{T}(y)
\end{equation*}
for $y$ near to $x$ since $\widehat{T}$ is linear near $x$, with $\widehat{T}'(x)=c_x$.
\end{proof}

Next corollary, whose proof is immediate,  allows us to characterize the local maxima of $T$.
Of course this result is well known, see \cite{AK} or \cite{K} for instance.

\begin{cor}
The Takagi function has a local maximum at $x$ if and only if $x\in \mathcal{A}$
and $c_x=0$.
\end{cor}

\begin{cor}
The Takagi function has empty superdifferential and
subdifferential almost everywhere.
\end{cor}

\vfill

\end{document}